\documentclass [12pt,a4paper]{amsart}

\usepackage{amssymb,amsmath,amsthm}

\usepackage{url}

\begin{document}

\newtheorem{theorem}{Theorem}[section]
\newtheorem{lemma}[theorem]{Lemma}
\newtheorem{proposition}[theorem]{Proposition}

\newtheorem{statement}[theorem]{Statement}
\newtheorem{conjecture}[theorem]{Conjecture}
\newtheorem{problem}[theorem]{Problem}
\newtheorem{corollary}[theorem]{Corollary}

\newtheorem*{maintheorem1}{Theorem A}
\newtheorem*{maintheorem2}{Theorem B}

\newtheoremstyle{neosn}{0.5\topsep}{0.5\topsep}{\rm}{}{\bf}{.}{ }{\thmname{#1}\thmnumber{ #2}\thmnote{ {\mdseries#3}}}
\theoremstyle{neosn}
\newtheorem{remark}[theorem]{Remark}
\newtheorem{definition}[theorem]{Definition}
\newtheorem{example}[theorem]{Example}

\numberwithin{equation}{section}

\newcommand{\Aut}{\,\mathrm{Aut}\,}
\newcommand{\Inn}{\,\mathrm{Inn}\,}
\newcommand{\End}{\,\mathrm{End}\,}
\newcommand{\Out}{\,\mathrm{Out}\,}
\newcommand{\Hom}{\,\mathrm{Hom}\,}
\newcommand{\ad}{\,\mathrm{ad}\,}
\newcommand{\SLn}{$\rm sl(n)$}
\newcommand{\calL}{{\mathcal L}}
\renewenvironment{proof}{\noindent \textbf{Proof.}}{$\blacksquare$}
\newcommand{\GL}{\,\mathrm{GL}\,}
\newcommand{\SL}{\,\mathrm{SL}\,}
\newcommand{\PGL}{\,\mathrm{PGL}\,}
\newcommand{\PSL}{\,\mathrm{PSL}\,}
\newcommand{\PSO}{\,\mathrm{PSO}\,}
\newcommand{\SO}{\,\mathrm{SO}\,}
\newcommand{\Sp}{\,\mathrm{Sp}\,}
\newcommand{\PSp}{\,\mathrm{PSp}\,}
\newcommand{\UT}{\,\mathrm{UT}\,}
\newcommand{\Exp}{\,\mathrm{Exp}\,}
\newcommand{\Rad}{\,\mathrm{Rad}\,}
\newcommand{\charr}{\mathrm{char}\,}
\newcommand{\tr}{\mathrm{tr}\,}
\newcommand{\diag}{{\rm diag}}
\newcommand \eps {\varepsilon}

\font\cyr=wncyr10 scaled \magstep1%

\def\Sha{\text{\cyr Sh}}

\title[Elementary equivalence of  automorphism groups of periodic Abelian groups]{Elementary equivalence of  endomorphism rings and automorphism groups of periodic Abelian groups}


\author{Elena Bunina}
\date{}
\address{Department of Mathematics, Bar--Ilan University, 5290002 Ramat Gan, ISRAEL}
\email{helenbunina@gmail.com}

\begin{abstract}
	 In this paper, we prove that the endomorphism rings $\End A$ and $\End A'$ of periodic infinite Abelian groups $A$ and $A'$ are elementarily equivalent if and only if the endomorphism rings of their $p$-components are elementarily equivalent for all primes $p$. Additionally, we show that the automorphism groups $\Aut A$ and $\Aut A'$ of periodic Abelian groups $A$ and $A'$ that do not have $2$-components and do not contain cocyclic $p$-components are elementarily equivalent if and only if, for any prime $p$, the corresponding $p$-components $A_p$ and $A_p'$ of $A$ and $A'$ are equivalent in second-order logic if they are not reduced, and are equivalent in second-order logic bounded by the cardinalities of their basic subgroups if they are reduced. According to~\cite{AVM=B-R}, for such groups $A$ and $A'$, their automorphism groups are elementarily equivalent if and only if their endomorphism rings are elementarily equivalent, and the automorphism groups of the corresponding $p$-components for all primes $p$ are elementarily equivalent.
\end{abstract}

\maketitle

\section{Introduction}\leavevmode

In this paper, we consider the elementary properties (i.e., properties expressible in first-order logic) of the endomorphism rings and automorphism groups of periodic Abelian groups.

The first to study the connection between the elementary properties of different models and the elementary properties of derivative models was A.I. Maltsev~\cite{Maltsev} in 1961. He proved that the groups $\mathcal G_n(K)$ and $\mathcal G_m(L)$, where $\mathcal G = \GL, \SL, \PGL, \PSL$ and $n, m \geqslant 3$, and $K, L$ are fields of characteristic $0$, are elementarily equivalent if and only if $m = n$ and the fields $K$ and $L$ are elementarily equivalent.

In 1992, this theory was further developed using ultraproduct constructions and the Keisler-Chang Isomorphism Theorem~\cite{Keisler} by K.I.\,Beidar and A.V.\,Mikhalev in the paper~\cite{BeiMikh}, where they found a general approach to problems of elementary equivalence among different algebraic structures and generalized Maltsev's theorem to the case where $K$ and $L$ are skew fields or prime associative rings.

The continuation of this research was carried out by E. Bunina (\cite{Bun2}--\cite{Bun5}, 1998--2019), where the results of A.I. Maltsev were extended to unitary linear groups over skew fields and associative rings with involutions, as well as to Chevalley groups over fields and commutative rings.

In general, there are numerous examples of broad classes of standard algebraic structures (such as fields or rings) and their classical derived structures (e.g., linear groups, Chevalley groups, etc.), where the elementary equivalence of these derived structures depends on matching both the parameters used in their construction (e.g., dimensions, root systems) and the elementary equivalence of the initial structures (such as fields or rings).

However, there are also more expressive derived structures for which elementary equivalence implies the stronger logical equivalence of the underlying structures.

We recall that  \emph{second-order logic} extends first-order logic by allowing quantification over relations, functions, and sets, not just individuals within the domain. This extended quantification enables second-order logic to express properties about structures that first-order logic cannot, providing a richer framework to analyze mathematical theories and structures. A \emph{second-order theory} of a structure captures all statements that hold true within the structure under this expanded framework. When considering \emph{second-order theories constrained by a cardinality}~$\varkappa$, we further limit the scope of second-order quantifiers to subsets or relations of this cardinality~$\varkappa$.

In 2000, Tolstykh examined the connection between second-order properties of skew fields and first-order properties of the automorphism groups of infinite-dimensional spaces over these skew fields \cite{Tolstyh}. He showed that these automorphism groups are elementarily equivalent if and only if the original skew fields are equivalent in second-order logic bounded by the cardinalities of dimensions of these spaces.

In 2003, Bunina and Mikhalev extended Tolstykh's results to categories of modules, endomorphism rings, automorphism groups, and projective spaces of modules of infinite rank over associative rings (see~\cite{categories}).

The famous Baer–Kaplansky Theorem (see~\cite{Fuks}) states that a periodic Abelian group is uniquely determined by its endomorphism ring; that is, if two periodic Abelian groups have isomorphic endomorphism rings, then they are also isomorphic. Similar results on isomorphic automorphism groups of Abelian 
$p$-groups were proven by Leptin and Libert for the cases 
$p\geq 5$ (see~\cite{Leptin}) and 
$p\geq 3$ (see~\cite{Liebert}), respectively (the case 
$p=2$ remains open). These are typical examples of initial structures (periodic Abelian groups or Abelian $p$-groups) and their derived structures (endomorphism rings or automorphism groups). Thus, it is natural to ask about the necessary and sufficient conditions under which these derived rings or groups are elementarily equivalent.

In~\cite{Abelian}, Bunina and Mikhalev established a connection between the second-order properties of Abelian 
$p$-groups and the first-order properties of their endomorphism rings. The works of Roizner \cite{My1}, \cite{My2}, and \cite{My3}, along with the final result of Bunina, Mikhalev, and Roizner~\cite{AVM=B-R}, focus on the elementary equivalence of automorphism groups of Abelian 
$p$-groups, specifically for $p\geq 3$.

The following theorem was ultimately proved:

\begin{theorem}[see~\cite{AVM=B-R}]
	The endomorphism rings $\End A$ and $\End A'$ of Abelian $p$-groups $A$ and $A'$, or the automorphism groups $\Aut A$ and $\Aut A'$ of Abelian $p$-groups $A$ and $A'$ with $p \geqslant 3$, are elementarily equivalent if and only if 
	
	\emph{1)} one of these groups is reduced and
	$
	Th_2^{\varkappa}(A) = Th_2^{\varkappa'}(A'),
	$
	where $\varkappa$ and $\varkappa'$ are the cardinalities of the basic subgroups of $A$ and $A'$, respectively;
	
	\emph{2)} if $A$ or $A'$ is not reduced, then 
	$
	Th_2(A) = Th_2(A').
	$
\end{theorem}

At the same time, \emph{weak second-order logic} (WSOL) is a variant of second-order logic in which quantification is limited to finite subsets of the domain, rather than allowing quantification over all possible subsets, as in full second-order logic. This restriction makes WSOL more manageable and closer in expressive power to first-order logic, while still allowing for the expression of certain structural properties that first-order logic cannot capture.

Several recent papers by various mathematicians have focused on a similar topic, specifically the interpretation of WSOL of an initial structure within the first-order logic of a derivative structure based on it.

For instance, in 2018, Kharlampovich and Miasnikov in~\cite{What-does} tackled the problem of elementary classification within the class of group algebras of free groups. We demonstrated that, unlike free groups, two group algebras of free groups over infinite fields are elementarily equivalent if and only if the groups are isomorphic and the fields are equivalent in weak second-order logic. Additionally, we established that the set of all free bases of a free group $F$ is $\emptyset$-definable in the group algebra $K(F)$ when $K$ is an infinite field. The set of geodesics is also definable, and many geometric properties of $F$ are definable in $K(F)$. Consequently, $K(F)$ encapsulates crucial information about~$F$. Similar results hold for group algebras of limit groups.

A significant development in this area was in 2022 the paper \cite{Myas-Kharl-Sohr} by Kharlampovich, Miasnikov, and Sohrabi, which examined groups and algebras where first-order logic has the same expressive power as weak second-order logic. These groups were termed \emph{rich groups}. Examples of such groups and algebras include $\GL_n(\mathbb{Z})$, $\SL_n(\mathbb{Z})$, and $T_n(\mathbb{Z})$ for $n \geq 3$, along with various finitely generated metabelian groups (such as free nonabelian ones), many polycyclic groups, free associative algebras, free group algebras over infinite fields, and others.

Many structures associated with a rich group (or ring) $G$, such as finitely generated subgroups (subrings, ideals), the geometry of its Cayley graph, and several additional features, are uniformly definable within it. By contrast, free and torsion-free hyperbolic groups are not rich --- they are, in fact, quite far from being rich. However, their group algebras over infinite fields are rich. This contrast illustrates just how much more expressive the first-order ring language of a group algebra of a free group is compared to the first-order language of the group itself.

Returning from weak second-order logic to more expressive fragments of second-order logic, we highlight a recent paper by Koberda and Gonzalez~\cite{Koberda}. In this work, the authors demonstrated that the first-order theory of the homeomorphism group of a compact manifold interprets the full second-order theory of countable groups of homeomorphisms of the manifold. Notably, this interpretation is uniform across manifolds with bounded dimension.

These recent findings have led the author to revisit the topic of interpreting the second-order logic of a structure within the first-order logic of certain naturally derived structures.

In this paper we achieve similar results for the endomorphism rings and automorphism groups of almost any periodic Abelian group.
 Since for $p=2$ all results about definability of $2$-groups by their automorphism groups are still open, considering automorphism groups of periodic Abelian groups we need to exclude $2$-components. Also in the case of the automorphism groups we will exclude cocyclic $p$-components (see Example 3.2).

So our result can be formulated as follows:

\begin{theorem}
	The endomorphism rings $\End A$ and $\End A'$ of arbitrary periodic Abelian groups $A=\bigoplus\limits_p A_p$ and $A'=\bigoplus\limits_p A_p'$ are elementarily equivalent if and only if, for any prime $p$, the endomorphism rings of the corresponding $p$-components $\End A_p$ and $\End A_p'$ are elementarily equivalent.
	
	The same is true for the automorphism groups $\Aut A$ and $\Aut A'$, provided that $A_2$ and $A_2'$ are trivial and that $A$ and $A'$ do not contain any nontrivial cocyclic $p$-components.
\end{theorem}

From the previous two theorems we have a direct corollary:

\begin{corollary}
	The endomorphism rings $\End A$ and $\End A'$ of arbitrary periodic Abelian groups $A = \bigoplus\limits_p A_p$ and $A' = \bigoplus\limits_p A_p'$ (or their automorphism groups $\Aut A$ and $\Aut A'$ in the case when $A_2$ and $A_2'$ are trivial and there are no nontrivial cocyclic components $A_p$ and $A_p'$) are elementarily equivalent if and only if for any prime $p$, for the corresponding $p$-components $A_p$ and $A_p'$:
	
	\emph{1)} One of these groups is reduced and
	$$
	Th_2^{\varkappa_p} (A_p) = Th_2^{\varkappa_p'} (A_p'),
	$$
	where $\varkappa_p$ and $\varkappa_p'$ are the cardinalities of the basic subgroups of $A_p$ and $A_p'$, respectively;
	
	\emph{(2)} If $A_p$ or $A_p'$ is not reduced, then 
	$$
	Th_2(A_p) = Th_2(A_p').
	$$
\end{corollary}

Since this corollary directly and immediately follows from Theorems 1.1 and 1.2, we will just prove Theorem~1.2 and deal with elementary equivalence without any special studies of second order logics.

In the last section we will give definitions concerning second order logic, some examples and facts about its expressible power and show that pairwise second order logical equivalence of $p$-components of our periodic groups doesn't imply second order equivalence of these groups themselves, so we {\bf cannot formulate} our main result as \emph{the endomorphism rings/automorphism groups of periodic Abelian groups is equivalent to the second order equivalence of the corresponding groups.}

In the next section we will show our result for endomorphism rings, the proof in this case is short and clear.

Then in the next sections we will do the same for automorphism groups.

\section{Elementary equivalence of endomorphism rings of periodic Abelian groups}\leavevmode

Let us assume that $ A = \bigoplus\limits_p A_p $ and $ A' = \bigoplus\limits_p A_p' $ are arbitrary periodic Abelian groups with the corresponding $ p $-components (subgroups containing all elements of orders $ p^n $, $ n \in \mathbb{N} \cup \{0\} $) $ A_p $ and $ A_p' $, where $ p $ is an arbitrary prime number.

In this case, 
\[
\End A = \prod_p \End A_p \quad \text{and} \quad \End A' = \prod_p \End A_p'.
\]

From \cite{Keisler}, we know that elementary equivalence is preserved under taking direct products and direct sums. This immediately implies that if for all primes $ p $ we have $ A_p \equiv A_p' $, then 
\[
\End A = \prod_p \End A_p \equiv \prod_p \End A_p' = \End A'.
\]

Therefore, we only need to prove the inverse implication, and it is sufficient to show that for every prime $ p $, the $ p $-component $ \End A_p $ is elementarily definable in $ \End A $ without parameters, i.e., we can find a formula $ \Phi_p (x) $ in the language of rings of one free variable such that 
$$
\forall r \in \End A \quad \End A \vDash \Phi(r) \Longleftrightarrow r \in \End A_p.
$$

To do this, we will use several well-known facts about endomorphism rings of Abelian groups (all of them can be found in~\cite{Fuks}).

\begin{enumerate}
	\item
	There exists a one-to-one correspondence between
	finite direct decompositions
	$$
	A = A_1 \oplus \dots \oplus A_n
	$$
	of the group $ A $ and decompositions of the ring $ \End(A) $
	in finite direct sums of left ideals
	$$
	\End(A) = L_1 \oplus \dots \oplus L_n;
	$$
	namely, if $ A_i = e_i A $, where $ e_1, \dots, e_n $ are mutually orthogonal
	idempotents, then $ L_i = \End(A)e_i $.
	
	\item
	An idempotent $ e \ne 0 $ is called \emph{primitive} if it cannot be represented
	as a sum of two nonzero orthogonal idempotents.
	If $ e \ne 0 $ is an idempotent of the ring $ \End(A) $, then $ eA $
	is an indecomposable direct summand
	of $ A $ if and only if $ e $ is a primitive idempotent.
	
	\item
	Let $ A = B \oplus C $ and $ A = B' \oplus C' $ be direct decompositions
	of the group $ A $, and let $ e \colon A \to B $ and $ e' \colon A \to B' $
	be the corresponding projections.
	Then $ B \cong B' $ if and only if there exist elements
	$ \alpha, \beta \in \End(A) $ such that
	$$
	\alpha \beta = e \text{ and } \beta \alpha = e'.
	$$
\end{enumerate}

\begin{theorem}[Charles \cite{Sharl1}, Kaplansky \cite{Kaplans}]\label{t2.9}
	The center of the endomorphism ring $\End(A)$ of a~$p$-group~$A$
	is the ring $J_p$ of $p$-adic integers or
	the residue class ring $\mathbb Z_{p^k}$ of the integers modulo~$p^k$,
	depending on whether $A$ is unbounded or bounded with $p^k$
	as the least upper bound for the orders of its elements.
\end{theorem}

\begin{corollary}
	For the group $ A = \bigoplus\limits_p A_p $, all central idempotents of the ring $ \End A $ are projections onto any direct summands $ B $ of $ A $ of the form 
	$$
	B_M = \bigoplus_{p \in M} A_p, \text{ where } M \text{ is some set of prime numbers}.
	$$
	
	Respectively, indecomposable central idempotents (i.e., not representable as a sum of two non-trivial central idempotents) are precisely projections onto $ A_p $.
\end{corollary}

\begin{proof}
	Since $ \End A = \prod\limits_p \End A_p $, the center of the ring $ \End A $ is the direct product of the centers of the corresponding components:
	$$
	Z(\End A) = \prod\limits_p Z(\End A_p).
	$$
	From Theorem~2.1, we know that $ Z(\End A_p) \cong J_p $ or $ \mathbb{Z}_{p^k} $. All these rings contain only trivial idempotents: $ 0 $ and $ 1 $, and therefore, if $ x = \prod\limits_p x_p \in Z(\End A) $, then for all $ p $ we have either $ x_p = 1 $ or $ x_p = 0 $.
	
	There is one exceptional case for the second statement: if the $ 2 $-component $ A_2 $ is bounded by the number $ 2 $, i.e., $ A_2 \cong \bigoplus\limits_\mu \mathbb{Z}_2 $. In this (and only in this) case, there exists a nontrivial idempotent $ e_2 $ such that $ e_2 + e_2 = 0 $. So in this case, we directly define $ e_2 $ and then work with $ e' = 1 - e_2 $, which is the projection onto all $ A_p $ except $ A_2 $. After that, any indecomposable central idempotent in the ring $ e' \End A $ is a projection onto some $ A_p $, $ p \ne 2 $.
\end{proof}

\medskip

So we can see that the set of all projections onto $ p $-components is definable. Now we want to define precisely the projection onto $ A_p $ for every concrete prime~$ p $. 

Let us denote the projection onto $ A_p $ by $ e_p $. 

Let us fix this $ p $ and note that for any natural $ n $, we can also define $ n \cdot e_p = e_p + \dots + e_p $, knowing the initial $ e_p $. We also note that if $ p $ and $ q $ are coprime, then there exists $ \alpha \in Z(\End A) $ such that $ \alpha \cdot (q e_p) = e_p $, since $ q $ is invertible in $ J_p $ or $ \mathbb{Z}_{p^k} $ in this case. On the other hand, $ p $ is not invertible in $ J_p $ or $ \mathbb{Z}_{p^k} $, so for $ p e_p $, we cannot find $ \alpha \in Z(\End A) $ such that $ \alpha \cdot (p e_p) = e_p $. 

Therefore, we can write down a formula defining (without parameters) the projection $ e_p : A \to A_p $ in $ \End A $:
\begin{multline*}
\mathbf{Proj}_p (x):=
(x^2=x) \land (\forall y (xy=yx))  \land ( \forall y (y\cdot px \ne x) \land \\
\land (\forall y \forall z (y^2=y\land z^2=z \land (\forall t (ty=yt \land tz=zt))
\land y\ne 0\land z\ne 0 \Rightarrow y+z\ne x)). 
\end{multline*}

Now all elements of $ \End A_p $ are those elements for which $ e_p $ is a unit: $ x \cdot e_p = x $.

Therefore, the first part of the main theorem is proved: \emph{endomorphism rings of periodic Abelian groups are elementarily equivalent if and only if endomorphism rings of their $ p $-components are respectively elementarily equivalent for all prime $ p $.}

\section{Involutions and extremal involutions in the groups $\Aut A$}\leavevmode

Next, in the three sections, we will study the automorphism groups of periodic Abelian groups. These sections are mostly devoted to studying involutions and extremal involutions in $ \Aut A $ for periodic Abelian groups. In the case of $ p $-groups, these involutions were used intensively in the papers \cite{My1}, \cite{My2}, \cite{My3}, and \cite{AVM=B-R} and were the most important tools for the proof of the main theorems in those works.

Further in this paper, we will consider periodic Abelian groups $ A = \bigoplus\limits_{p \ne 2} A_p $, which are direct sums of their $ p $-components $ A_p $, where $ A_2 = 0 $.

For torsion groups $ A $, our first concern is to identify the center of $ \Aut A $.

We may restrict our considerations to $ p $-groups, $ p \ne 2 $.

\begin{theorem}[Baer~\cite{Baer}]
	If $p\ne 2$, then the center of the automorphism group of a
$p$-group~$A$ consists of

\emph{(1)}  multiplications by $p$-adic units, if $A$ is unbounded, and

\emph{(2)} multiplications by integers $k$ with $1\leqslant  k < p^n$ coprime to~$p$, if $A$ is bounded, and
$p^n$ is the smallest bound.
\end{theorem}

Therefore, the center of $ \Aut A $ of our periodic $ A $ is the direct product of the corresponding centers of $ \Aut A_p $, as described in the previous theorem.

Let us mention that if $ A_p \cong \mathbb{Z}_{p^k} $ or $ \mathbb{Z}_{p^\infty} $ (these two types together are called \emph{cocyclic} Abelian $ p $-groups, see~\cite{Fuks}), then $ \Aut A_p $ is abelian and coincides with its center. In all other cases, it is non-Abelian, contains at least two independent cocyclic direct summands, and respectively contains different involutions which correspond to direct decompositions of $ A_p $.

If our periodic Abelian group contains several $ p $-components isomorphic to cocyclic groups, then we generally cannot define it by its automorphism group. Let us show a simple example demonstrating this:

 \begin{example}
 Let us take two periodic Abelian groups: $ A_1 = \mathbb{Z}_3 \oplus \mathbb{Z}_{13} $ and $ A_2 = \mathbb{Z}_5 \oplus \mathbb{Z}_7 $. Their automorphism groups are:
 $$
 \Aut A_1 = \Aut (\mathbb{Z}_3) \oplus \Aut (\mathbb{Z}_{13}) = \mathbb{Z}_2 \oplus \mathbb{Z}_{12} \cong \mathbb{Z}_2 \oplus \mathbb{Z}_4 \oplus \mathbb{Z}_3
 $$
 and 
 $$
 \Aut A_2 = \Aut (\mathbb{Z}_5) \oplus \Aut (\mathbb{Z}_{7}) = \mathbb{Z}_4 \oplus \mathbb{Z}_{6} \cong \mathbb{Z}_2 \oplus \mathbb{Z}_4 \oplus \mathbb{Z}_3,
 $$
 so we have two periodic non-isomorphic finite Abelian groups with isomorphic automorphism groups.
 \end{example}

Since for periodic Abelian groups with cocyclic $ p $-components, their automorphism groups often do not define these groups, we can assume that all $ p $-components of our groups $ A_1 $ and $ A_2 $ are not cocyclic (and therefore have non-trivial direct summands).

\bigskip

Any involution $ \eps $ in our group $ A $ without $ 2 $-components corresponds to a decomposition of the group $ A $ into a direct sum $ A = A_\eps^+ \oplus A_\eps^- $, where $ A_\eps^+ = \{ a \in A \mid \eps a = a \} $ and $ A_\eps^- = \{ a \in A \mid \eps a = -a \} $. Of course, $ A_\eps^+ = \bigoplus\limits_{p \ne 2} A_{p,\eps}^+ $ and $ A_\eps^- = \bigoplus\limits_{p \ne 2} A_{p,\eps}^- $.

From \cite{Fuks2015}, page 659, we know that for $ p $-groups, \emph{two involutions $ \theta, \xi $ of $ A $ commute if and only if
	$$
	A_\theta^- = (A_\theta^- \cap A_\xi^-) \oplus (A_\theta^- \cap A_\xi^+) \text{ and } A_\theta^+ = (A_\theta^+ \cap A_\xi^-) \oplus (A_\theta^+ \cap A_\xi^+)
	$$
	and similarly for $ A_\xi^- $, $ A_\xi^+ $}.

If two involutions $ \theta, \xi $ of a periodic Abelian group $ A $ commute, then their corresponding involutions $ \theta_p, \xi_p $ of the $ p $-parts of $ A $ also commute. For every part $ A_p $, we have similar decompositions, and therefore we have globally this decomposition for $ \theta $ and $ \xi $. It is clear that the same is true in the opposite direction.

The following statement from \cite{Fuks2015}, page 659, is also true for periodic groups without $ 2 $-components:

\emph{A direct decomposition
$$
A \cong  C_1\oplus \dots \oplus C_n, \quad C_i\ne 0,
$$
defines a set  $\{ \theta_1,\dots, \theta_n\}$ of commuting involutions such that $\theta_i\vert_{C_j} =\delta_{ij} id$. Conversely, a set $\{ \theta_1,\dots, \theta_n\}$ of pairwise commuting involutions determines a unique decomposition $A \cong  C_1\oplus \dots \oplus C_n$.}

\begin{definition}
	An involution $\eps$ is said to be \emph{extremal} if either $A_\eps^+$ or $A_\eps^-$ is indecomposable.
	\end{definition}
	
	\medskip
	
	From  the~book~\cite{Fuks} (vol.~2, p.~310) we know that for any Abelian $p$-group $A$, $p\ne 2$:
	
	\emph{An involution $\eps$ is extremal if and only if the~group~$Z(C\{ \eps, \zeta\})$ contains not more than $8$ involutions for all~$\zeta\in C\{ \eps\}$.}
	
	\medskip
	
	Now if $ A = \bigoplus\limits_{p \ne 2} A_p $, then for any involution $ \eps \in \Aut A $, if $ \zeta $ commutes with $ \eps $, then for any prime $ p $, the corresponding involution $ \eps_p $ in $ \Aut A_p $ commutes with the corresponding $ \zeta_p \in \Aut A_p $. If we denote the corresponding $ Z(C(\eps_p, \zeta_p)) \in \Aut A_p $ by $ ZC_p $, then $ Z(C\{ \eps, \zeta \}) = \prod\limits_p ZC_p $.
	
	Let us consider different cases:
	
	\smallskip
	
	(1) If on $ A_p $ both involutions $ \eps_p $ and $ \zeta_p $ are central, then $ C(\eps_p, \zeta_p) = \Aut(A_p) $ and $ ZC_p $ contains only central involutions of $ \Aut A_p $, i.e., $ \pm 1 $ on $ A_p $.
	
	\smallskip
	
	(2) If on $ A_p $ one of the involutions $ \eps_p $ or $ \zeta_p $ is central while the other is not, then any involution from $ ZC_p $ is either central or this non-trivial $ \eps_p $ or $ \zeta_p $ multiplied by $ \pm 1 $.
	
	\smallskip
	
	(3) If on $ A_p $ both $ \eps_p $ and $ \zeta_p $ are not central (and commute), then in the case of extremal $ \eps_p $, the subgroup $ ZC_p $ contains not more than $ 8 $ involutions in $ \Aut A_p $ and not more than $ 4 $ modulo center. If $ \eps_p $ is not extremal and not central on $ A_p $, then we can find $ \zeta_p $ on $ A_p $ such that the corresponding $ ZC_p $ contains at least $ 8 $ different involutions modulo center.
	
		\smallskip
	
Therefore, we see that if $ \eps_p $ is non-central and non-extremal for at least one $ p $, then there exists an involution $ \zeta \in C(\eps) $ such that $ Z(C(\eps, \zeta)) $ contains more than $ 4 $ involutions modulo $ Z(\Aut A) $.

If $ \eps $ is extremal and non-central in at least two components, then there exist two primes $ p \ne q $ such that $ \eps_p $ and $ \eps_q $ are extremal. We can take a corresponding $ \zeta $ such that $ \zeta_p $ and $ \zeta_q $ are extremal, commute with $ \eps_p $ and $ \eps_q $, and differ from the corresponding $ \eps_p $ and $ \eps_q $ modulo center (we can do this since all components are not cocyclic). In this case, in each component $ \Aut A_p $ and $ \Aut A_q $, we have $ 4 $ involutions modulo the corresponding centers in $ ZC_p $ and $ ZC_q $. This means that we have at least $ 8 $ involutions modulo the center $ Z(\Aut A) $ in $ Z(C(\eps, \zeta)) $.

So we see that an extremal involution in $ \Aut A $ can be defined by the formula

\begin{multline*}
Extreme (\eps):= (\eps \ne e) \land (\eps^2=e)\land\\
\land  \forall \zeta \,  \left(\zeta^2=e \land \eps \zeta = \zeta \eps \Longrightarrow \forall \xi_1,\xi_2,\dots, \xi_5 \, \left( \left(\bigwedge_{i=1}^5 \xi_i^2=e\right) \land  \right.\right.\\
 \left. \left. \land \forall x \left((x\eps =\eps x\land x \zeta = \zeta x)
\Rightarrow  \bigwedge_{i=1}^5 \xi_i x=x\xi_i\right) \Rightarrow \left(\bigvee_{i\ne j} \exists z (\forall y (zy=yz)\land \xi_i=z\xi_j) \right) \right) \right).
\end{multline*}
	
	This formula~$Extreme(\eps)$ means that the automorphism~$\eps$ is an extreme involution (i.\,e., an involution which has one of its summands~$A_\eps^+$ or $A_\eps^-$ to be indecomposable).
The indecomposable summand for the~involution~$\eps$ is denoted by~$A_\eps$, while the other summand is denoted by~$A_\eps^\perp$. Saying ``$\eps$ acts on~$A_p$'' or ``$\eps$ belongs to~$A_p$'' we will mean $A_\eps \subset A_p$.

\begin{lemma}[see \cite{My1}]
	For a periodic Abelian group $A$ without $2$-components  and cocyclic $p$-components the formula
\begin{multline*} \eps \in \eps_1 \oplus \eps_2 :=\\
	=\forall \eps' \Bigl(Extreme(\eps') \land
		(\eps'\eps_1 =
		\eps_1\eps' )\land (\eps'\eps_2 = \eps_2\eps' )
		\land (\eps'\ne\eps_1 )\land ( \eps'\ne\eps_2) \Rightarrow ( \eps\eps' =
		\eps'\eps)\Bigr),
\end{multline*}
	for extremal involutions $\eps$, $\eps_1$, $\eps_2$ such that $\eps_1
	\eps_2 = \eps_2 \eps_1$, means that $A_\eps \subset A_{\eps_1}
	\oplus A_{\eps_2}$, and $A_\eps^\perp\supset A_{\eps_1}^\perp \cap
	A_{\eps_2}^\perp$.
\end{lemma}
\begin{proof}
If $ \eps, \eps_1 $, and $ \eps_2 $ belong to the same $ p $-component, then this statement follows from \cite{My1}.

If $ \eps_1 $ and $ \eps_2 $ belong to different components, for example, a $ p $-component and a $ q $-component, then for an extremal involution $ \eps $ to have $ A_\eps \subset A_{\eps_1} \oplus A_{\eps_2} $ means it must coincide with either $ \eps_1 $ or $ \eps_2 $. In this case, $ \eps' $ under consideration automatically commutes with $ \eps $. Conversely, if we know that every extremal involution commuting with $ \eps_1 $ from the $ p $-component and with $ \eps_2 $ from the $ q $-component also commutes with $ \eps $, it means that $ \eps $ is from the $ p $- or $ q $-component. If it is (for example) from the $ p $-component, then every extremal involution from the $ p $-component commuting with $ \eps_1 $ must commute with $ \eps $, therefore $ \eps = \eps_1 $.

Therefore, in this case, $ A_\eps \subset A_{\eps_1} \oplus A_{\eps_2} $ and $ A_\eps^\perp \supset A_{\eps_1}^\perp \cap A_{\eps_2}^\perp $ if and only if the corresponding formula $ \eps \in \eps_1 \oplus \eps_2 $ holds.

The last case is if $ \eps_1 $ and $ \eps_2 $ both belong to the same $ p $-component and $ \eps $ belongs to another $ q $-component. In this case, any involution $ \eps' $ from the $ q $-component commutes with $ \eps_1 $ and $ \eps_2 $. So if we take any extremal involution $ \eps' $ from the $ q $-component, which does not commute with $ \eps $, our formula $ \eps \in \eps_1 \oplus \eps_2 $ will not hold.

Therefore, we have considered all cases and proved the lemma.
\end{proof}

\medskip

\begin{lemma}
	Two extremal involutions $ \eps_1 $ and $ \eps_2 $ belong to different components $ A_p $ and $ A_q $ (i.e., $ A_{\eps_1} \subset A_p $ and $ A_{\eps_2} \subset A_q $, $ p \ne q $) if and only if
	\begin{multline*}
		\text{InDiffComp}(\eps_1,\eps_2) := \text{Extreme}(\eps_1) \land \text{Extreme}(\eps_2) \land \\
		\land (\eps_1 \eps_2 = \eps_2 \eps_1) \land \forall \eps \Bigl((\eps \in \eps_1 \oplus \eps_2) \land \text{Extreme}(\eps) \Rightarrow \exists z(\forall x (xz = zx) \land (\eps = \eps_1 \cdot z \lor \eps = \eps_2 \cdot z))\Bigr).
	\end{multline*}
	Therefore, there exists a formula $ \text{InOneComp}(\eps_1, \eps_2) $ which holds on extremal involutions $ \eps_1 $, $ \eps_2 $ if and only if $ \eps_1 $ and $ \eps_2 $ belong to the same $ p $-component.
\end{lemma}

\begin{proof}
	The proof directly follows from the proof of the previous lemma.
\end{proof}

\medskip

We are also interested in involutions which are not necessarily extremal, not central in the whole $ \Aut A $ and are central on all $ q $-components $ \Aut A_p $ except one given $ p $-component. 

Now we want to distinguish them from other arbitrary involutions (this class includes extremal involutions).

Let us call them \emph{normal} involutions.
  
  \begin{lemma}
  A non-central  involution $\xi$ is normal if and only if it satisfies the formula
\begin{multline*}
  Normal (\xi): = \exists \eps \Bigl(Extreme(\eps) \land (\eps \xi =\xi \eps) \land \\
  \land \exists \eps_1 \bigl(InOneComp(\eps,\eps_1) \land \eps_1 \xi \ne \xi \eps_1\bigr) \land \forall \eps_2 \bigl(InDiffComp (\eps,\eps_2)\Rightarrow \eps_2 \xi=\xi \eps_2\bigr)\Bigr).
\end{multline*}

Respectively, a normal involution $\xi$ and an extremal involution $\eps$ correspond to the same $p$-component~$A_p$ if and only if they satisfy the formula
$$
NECorresp (\xi,\eps):= Normal(\xi) \land Extreme (\eps) \land \exists \eps' (InOneComp(\eps,\eps') \land \eps' \xi \ne \xi \eps').
$$
  \end{lemma} 
  
  \begin{proof}
  	If an involution $\xi$ is not central, then it is not central on at least one component~$A_p$. Therefore there exists at least one extremal involution on~$A_p$ that does not commute with~$\xi$. If $\xi$ is not central also on~$A_q$, then there exist two exstremal involutions from different components that do not commute with~$\xi$. Therefore if the formula $Normal(\xi)$ holds, then $\xi$ is not central presicely on one component, i.\,e., is normal. Conversely, if $\xi$ is normal, then $Normal (\xi)$ holds.
  	
  	The second statement is evident. 
  \end{proof}
  
  \medskip 

With only a normal involution~$\xi$, we cannot distinguish the groups $A_\xi^+\cap A_p$ and $A_\xi^-\cap A_p$ by the first order language. Therefore we consider pairs~$(\xi, \eps)$ with the condition~$NECorresp(\xi,\eps)\land(\xi\eps = \eps\xi)$. For such pairs, either $A_\eps \subset A_\xi^+\cap A_p$ or $A_\eps \subset A_\xi^-\cap A_p$. Therefore $A_\eps$ indicates the required group among~$A_\xi^+\cap A_p$ and~$A_\xi^-\cap A_p$ (it is denoted by~$A_{(\xi, \eps)}$). The property of being a pair is denoted by the formula~$Pair(\xi, \eps)$. Instead of $\forall \xi \forall \eps (Pair(\xi, \eps) \Rightarrow (\dots))$ and $\exists \xi \exists \eps (Pair(\xi, \eps) \land (\dots))$, we will write $\forall (\xi, \eps)$ and $\exists(\xi, \eps)$, respectively.

\section{Defining $p$-components up to the center}\leavevmode

Since our goal is to prove that in the first order logic of the group $ \Aut A $ we can interpret the second order logic of $ A_p $ for every corresponding $ p $, we need for every prime $ p \ne 2 $ to define something as close as possible to the first order theory of $ \Aut A_p $.

Therefore, it is important for each prime $ p \ne 2 $ to define whether a given extremal involution belongs to $ A_p $, or not.

\medskip


\begin{lemma} In $\Aut A$ for a periodic Abelian group~$A$ without $2$-components and cocyclic components 
for any commuting extremal involutions $\eps_1$ and $\eps_2$ from one $p$-component the set
$$
D_{\eps_1,\eps_2} : = \{ x\mid x\in Z(C(\eps_1,\eps_2))\}
$$
consists of all automorphisms from $\Aut A$ that are central on all $q$-components, $q\ne p$, multiply the direct summand $A_{\eps_1}$ by some $\alpha \in \Aut A_{\eps_1}$, the direct summand $A_{\eps_2}$ --- by some  $\beta \in \Aut A_{\eps_3}$ and the direct summand $A_{\eps_1}^\perp \cap A_{\eps_2}^\perp \cap A_p$ --- by some $\gamma$ from the corresponding center.

Modulo the center $Z(\Aut A)$ we can consider these automorphisms as ''diagonal matrices'' $2\times 2$ of the form $\diag [\alpha, \beta]$, $\alpha \in \Aut A_{\eps_1}$, $\beta \in \Aut A_{\eps_2}$.
\end{lemma}

\begin{proof}
	Clear.
\end{proof}

\medskip

Now for simplicity we will concentrate on the direct summand $A_{\eps_1} \oplus A_{\eps_2}$ and will consider only such automorphisms that belong to $Z(C(\eps_1\eps_2))$, i.\,e. are arbitrary on $A_{\eps_1} \oplus A_{\eps_2}$ and central on $(A_{\eps_1} \oplus A_{\eps_2})^\perp$. Since all our automorphisms we can consider only modulo center, we will also consider these arbitrary automorphisms of $A_{\eps_1} \oplus A_{\eps_2}$ modulo center.

These automorphisms can be represented by ''invertible matrices'' 
$$
\begin{pmatrix}
\alpha & \beta\\
\gamma & \delta
\end{pmatrix},\quad \alpha \in \End A_{\eps_1}, \beta \in \Hom(A_{\eps_2}, A_{\eps_1}), \gamma \in \Hom(A_{\eps_1}, A_{\eps_2}), \delta \in  \End A_{\eps_2}.
$$
Here the ring $\End A_{\eps_i}$, $i=1,2$,  is isomorphic to the ring $\mathbb Z_{p^n}$ in the case $A_{\eps_i} \cong \mathbb Z_{p^n}$ or to the ring $J_p$ in the case $A_{\eps_i} \cong \mathbb Z_{p^\infty}$.  In the case when both $A_{\eps_1}$ and $A_{\eps_2}$ are $\mathbb Z_{p^\infty}$, we can assume that we deal with the matrices from $M_2(J_p)$. If one of the groups is $\mathbb Z_{p^\infty}$ and another is $\mathbb Z_{p^n}$, then $\Hom (\mathbb Z_{p^\infty}, \mathbb Z_{p^n})=0$, $\Hom (\mathbb Z_{p^n}, \mathbb Z_{p^\infty})\cong \mathbb Z_{p^n}$. If one the groups is $\mathbb Z_{p^n}$ and another is $\mathbb Z_{p^m}$, $n\geqslant m$, then $\Hom (\mathbb Z_{p^n}, \mathbb Z_{p^m}) \cong \mathbb Z_{p^m}$, $\Hom (\mathbb Z_{p^m}, \mathbb Z_{p^n}) \cong \mathbb Z_{p^m}$ and consists of elements from $\mathbb Z_{p^n}$ dividing by~$p^{n-m}$.


\begin{lemma}
	If under conditions of the previous lemma an automorphism $\mu \in \Aut A$ acting centrally on $(A_{\eps_1} \oplus A_{\eps_2})^\perp$, satisfies the condition 
	$$
	\eps_1 \mu \eps_1 \mu =1,
	$$
	then the corresponding matrix $\begin{pmatrix} \alpha & \beta\\ \gamma & \delta \end{pmatrix}$ satisfies the following conditions:
	 $$
	 \alpha^2-\beta \gamma =1,\quad \delta^2 - \gamma \beta=1,\quad \gamma \alpha =\delta \gamma,\quad \alpha \beta =\beta \delta.
	 $$
\end{lemma}

\begin{proof}
	Direct cheking.
\end{proof}

\medskip

\begin{lemma}
	Assume that under conditions of the previous lemmas an automorphism $\mu \in \Aut A$ acts centrally on $(A_{\eps_1} \oplus A_{\eps_2})^\perp$ and satisfies the  additional condition 
	$$
QTransv(\mu):=(\mu^2\ne 1) \land	(\eps_1 \mu \eps_1 \mu =1) \land (\forall d \in D_{\eps_1,\eps_2} \  (d\mu d^{-1}) \mu= \mu (d\mu d^{-1})).
	$$
	Then if $\eps_1$ and $\eps_2$ both belong to $A_p$, $p\ne 3$, then for all prime numbers $q\ne p$ there exists a root of the $q$-th power from~$\mu$ and for some such $mu$ it does not have a root of the $p$-th power.
	
	If $p=3$, then for some such $\mu$ it has a root of the $q$-th power, $q\ne 3$,  and does not have a root of the $3$-th power.
\end{lemma}

\begin{proof}
	Let us assume that on $A_{\eps_1} \oplus A_{\eps_2}$ our $\mu$ has the matrix $\begin{pmatrix} \alpha & \beta\\ \gamma & \delta \end{pmatrix}$. Assume also that we take $d=\diag [\lambda ,1]$.
	
	In this case we have 
	$$
	\begin{pmatrix}
		 \alpha & \lambda \beta\\ 
		 \gamma  \cdot \lambda^{-1} & \delta \end{pmatrix}
		  \cdot \begin{pmatrix} \alpha & \beta\\ \gamma & \delta \end{pmatrix} =
		   \begin{pmatrix} \alpha & \beta\\ \gamma & \delta \end{pmatrix} \cdot 
		   \begin{pmatrix}
		   	\alpha & \lambda \beta\\ 
		   	\gamma  \cdot \lambda^{-1} & \delta \end{pmatrix}.
		   	$$
		   	
		   	It means 
		   	$$
		   	\alpha^2 + \lambda \beta \gamma = \alpha^2 + \beta \gamma \lambda^{-1} \text{ and }
		   	\delta^2 +  \gamma \lambda^{-1} \beta = \delta^2 + \gamma \lambda \beta,
		   	$$
		   	which immediately gives us $(\lambda^2 -1) \beta \gamma=0=\gamma (\lambda^2-1) \beta$ for all invertible $\lambda$ from $\End A_{\eps_1}$. 
		   	
		   	If $p\ne 3$, it means $\beta \gamma = 0 =\gamma \beta$, since for all $p\geqslant 5$ there exists a central invertible element $\lambda$ such that $\lambda^2-1$ is invertible.
		   	
		   	In this case from the previous lemma $\alpha^2 = \delta^2=1$ and therefore $\alpha=\pm 1$, $\delta =\pm 1$. In this case $\alpha =-\delta$ is impossible, because then $\gamma =- \gamma$ and $\beta  = -\beta$, i.\,e., $\gamma = \beta =0$ and $\mu^2=1$. Since we consider all automorphisms modulo center, we can assume that $\alpha = \delta =1$. So now we have 
		   	$$
		   	\mu = \begin{pmatrix}
		   		1& \beta \\
		   		\gamma & 1
		   		\end{pmatrix},\quad \beta \gamma=\gamma\beta =0.
		   		$$
		   		
		   		Let us find the $n$-th power of~$\mu$:
		   		$$
		   		\begin{pmatrix}
		   			1& \beta \\
		   			\gamma & 1
		   			\end{pmatrix}^2=\begin{pmatrix}
		   			1& 2\beta \\
		   			2\gamma & 1
		   			\end{pmatrix}, \dots, \begin{pmatrix}
		   			1& \beta \\
		   			\gamma & 1
		   			\end{pmatrix}^n=\begin{pmatrix}
		   			1& n\beta \\
		   			n\gamma & 1
		   			\end{pmatrix}.
		   			$$
		   			It means that if we are in the $p$-component $A_p$ of~$A$, $p\ne 3$, then all  such~$\mu$ have a root of the power~$q$  for all $q\ne p$ and there exist some such $\mu$ that does not have a root of the power~$p$ (since in all $\mathbb Z_{p^n}$ and in $J_p$ all elements are divided by all $q\ne p$ and there are elements which are not divided by~$p$).
		   			
		   			If $p=3$, then of course there exsits a corresponding $\mu$ (with $\gamma=0$, for example) which  has all roots of the power $q\ne 3$ and does not have a root of the power~$3$.   			
\end{proof}

\medskip

Let us call all automorphisms $\mu$ acting centrally on $(A_{\eps_1}\oplus A_{\eps_2})^\perp$ and satisfying the formula $QTransv_{\eps_1,\eps_2}(\mu)$, \emph{quasi-transvections} on $\eps_1$ and $\eps_2$.

\begin{lemma}
	Assume that $A=\bigoplus\limits_{p\ne 2} A_p$, where all $p$-components $A_p$ are not cocyclic.
	
	For a given prime $p\ne 3$ a given extremal involution $\eps$ on~$A$ belongs to~$A_p$ if and only if it satisfies the formula
\begin{multline*}
	Comp_p(\eps):= \exists \eps_1, \eps_2 (InOneComp(\eps,\eps_1) \land InOneComp (\eps,\eps_2)\land\\
	\land  (\eps_1\eps_2=\eps_2\eps_1) \land \forall \mu (QTransv_{\eps_1,\eps_2}(\mu) \Rightarrow \exists \nu (QTransv_{\eps_1,\eps_2}(\nu) \land \nu^3=\mu) )  \land\\
	\land  \exists \mu (QTransv_{\eps_1,\eps_2}(\mu) \land \forall \nu (QTransv_{\eps_1,\eps_2}(\nu) \Rightarrow  \nu^p \ne\mu) )) ;
\end{multline*}
 a given extremal involution $\eps$ on~$A$ belongs to~$A_3$ if and only if it satisfies the formula
 \begin{multline*}
 	Comp_3(\eps):= \exists \eps_1, \eps_2 (InOneComp(\eps,\eps_1) \land InOneComp (\eps,\eps_2)\land\\
 	\land  (\eps_1\eps_2=\eps_2\eps_1) \land \\
 	\land  \exists \mu (QTransv_{\eps_1,\eps_2}(\mu) \land \forall \nu (QTransv_{\eps_1,\eps_2}(\nu) \Rightarrow  \nu^3 \ne\mu) )) .
 \end{multline*}
\end{lemma}

\begin{proof}
	Directly follows from the previous lemma.
\end{proof}

\section{Conclusions}\leavevmode

Now we see that if we have a periodic Abelian group $ A = \bigoplus\limits_{p \ne 2} A_p $, where all $ A_p $ are not cyclic, then for each $ p \ne 2 $ the following are first order interpretable in the group $ \Aut A $:

\begin{enumerate}
	\item the group $ \Aut A_p $ up to the center of $ \Aut \bigoplus\limits_{q \ne p} A_q $;
	\item the group $ \Aut A_p / Z(\Aut A_p) $;
	\item the set of all extremal involutions acting on $ A_p $;
	\item the set of all normal involutions acting on $ A_p $.
\end{enumerate}

For every prime $ p \ne 2 $, this set of objects is absolutely sufficient for interpretation of the full second order theory in the case when $ A_p $ is not reduced and of the second order theory bounded by the cardinality of a basic subgroup of $ A_p $, if $ A_p $ is reduced. For reduced Abelian $ p $-groups, it was done in the papers \cite{My1} and \cite{My2}, for non-reduced Abelian $ p $-groups --- in the paper \cite{My3}, and then the joint result was published in the PhD thesis of M.\,Roizner \cite{Misha-phd}.

Since in these papers it was also proved that if for two Abelian $ p $-groups $ A $ and $ A' $ from their second order equivalence it follows elementary equivalence of their automorphism groups, and in the case of reduced $ A $ and $ A' $, from $ Th_2^{|B|} (A) = Th_2^{|B'|} (A') $ it follows $ \Aut A \equiv \Aut A' $, then if for our $ A = \bigoplus\limits_{p \ne 2} A_p $ and $ A' = \bigoplus\limits_{p \ne 2} A_p' $ without cocyclic $ p $-components we have $ \Aut A \equiv \Aut A' $, then we have $ Th_2(A_p) = Th_2(A_p') $ or $ Th_2^{|B|} (A_p) = Th_2^{|B'|} (A_p') $ for all prime $ p $ and therefore $ \Aut A_p \equiv \Aut A_p' $ for all prime $ p $.

Consequently, finally Theorem 1.2 is proved.

\begin{remark}
	We promised in the beginning of this paper to show that even if $A_p\equiv_2 A_p'$ for all prime~$p$, it does not automatically imply $\bigoplus\limits_p A_p \equiv_2 \bigoplus\limits_p A_p'$.
	
	Actually, let us take two cardinal numbers $\varkappa_1$ and $\varkappa_2$ such that $\varkappa_1< \varkappa_2$ and $\varkappa_1 \equiv_2 \varkappa_2$ (for examples of such cardinals see~\cite{BunBrag}). Let now $A_3=\bigoplus\limits_{\varkappa_1} \mathbb Z_3$, $A_3'=\bigoplus\limits_{\varkappa_2} \mathbb Z_3$, $A_5=A_5'=\bigoplus\limits_{\varkappa_2} \mathbb Z_5$, $A=A_3\oplus A_5$ and $A'=A_3'\oplus A_5'$.
	
	So $A_3\equiv_2 A_3'$, $A_5\equiv_2 A_5'$, but $A\not\equiv_2 A'$, since we can write a second order formula formulating that the $3$-component of~$A$ has a cardinality smaller than~$A$.
\end{remark}


\begin{thebibliography}{99}
	





\bibitem{Baer}
Baer R.
Automorphism rings of primary abelian operator groups.
Ann. Math., 44 (1943), 192--227.




\bibitem{BeiMikh}
Beidar C.I. and Michalev A.V. \emph{On Malcev's theorem on elementary
equivalence of linear groups.}
Contemporary mathematics, 1992, {\bf 131}, 29--35.




\bibitem{BunBrag} Bragin V.A., Bunina E.I. \emph{An example of two cardinals that are equivalent in the $n$-order logic and not equivalent in the $(n + 1)$-order logic}, J. Math. Sci., 2014, {\bf 201}, 431--437.



\bibitem{Bun2}
Bunina E.I.
\emph{Elementary equivalence of unitary linear groups over rings and skewfields},
Russian Mathematical Surveys, 1998, {\bf 53}(2), 137--138.

\bibitem{Bun3}
Bunina E.I., Mikhalev A.V. \emph{Combinatorial and logical aspects of linear groups and Chevalley groups}. Acta Applicandae Mathematicae, 2005, {\bf 85}(1--3), 57--74.

\bibitem{Bun4}
Bunina E.I.
\emph{Elementary equivalence of Chevalley groups over local rings.}
Sbornik: Mathematics, 2010, {\bf 201}(3), 3--20.

\bibitem{Bun5} Bunina E.I. \emph{Isomorphisms and elementary equivalence of Chevalley groups over commutative rings}. Sbornik: Mathematics, 2019, {\bf 210}(8), 1067--1091.

\bibitem{Abelian} Bunina E.I., Mikhalev A.V. \emph{Elementary equivalence of endomorphism rings Abelian $p$-groups}, J. Math. Sci., 2006, {\bf 137}(6), 5212--5274.

\bibitem{categories} Bunina E.I., Mikhalev A.V. \emph{Elementary equivalence
of categories of modules over rings, endomorphism rings and automorphism
groups of modules}, J. Math. Sci.,  2006, {\bf 137}(6), 5275--5335.

\bibitem{My1}
Bunina E.I., Roizner M.A.
\emph{Elementary equivalence of the automorphism groups of Abelian $p$-groups,}
J.  Math. Sci., 2010, {\bf 169}(5), 614--635.

\bibitem{AVM=B-R} Bunina E.I., Mikhalev A.V., Roizner M.A.  \emph{The criteria of elementary equivalence of automorphism groups and endomorphism rings of Abelian $p$-groups}, Dokl. Ross. Akad. Nauk, 2014, {\bf 457}(1), 11--12.


\bibitem{Keisler} Chang C.C. and Keisler H.G. \emph{Model theory}. American Elsevier Publishing Company, inc., New York, 1973, 3d Edition, 1990.





\bibitem{Sharl1}
Charles B. \emph{Le centre de l'anneau des endomorphismes
d'un groupe ab\'elien primaire},
C.~R.~Acad.\ Sci.\ Paris, 1953, {\bf 236}, 1122--1123.





\bibitem{Fuks}
Fuchs L.
\emph{Infinite abelian groups.
Volumes} I, II,  Tulane University New Orleans, Louisiana, 1970.



\bibitem{Fuks2015}
Fuchs L. \emph{Abelian groups}, Springer International Publishing Switzerland, 2015.









\bibitem{Kaplans}
Kaplansky I.
\emph{Infinite abelian groups,}
University of Michigan Press., Ann. Arbor, Michigan, 1954 and 1969.

\bibitem{What-does} Kharlampovich O., Myasnikov A. \emph{What does a group algebra of a free group “know” about the group?}. Annals of Pure and Applied Logic, 2018, {\bf 169}(6), 523--547.




\bibitem{Myas-Kharl-Sohr} Kharlampovich O., Myasnikov A., Sohrabi M. \emph{Rich groups, weak second order logic, and applications}. In ``Groups and Model Theory'', de Gruyter, 2021, 127--193.

\bibitem{Koberda} Koberda T.,  González J. \emph{Uniform first order interpretation of the second order theory of countable groups of homeomorphisms}, 2024, https://arxiv.org/abs/2312.16334, 32 pp.


\bibitem{Leptin}
Leptin H.
\emph{Abelsche $p$-Gruppen und ihre Automorphismengruppen},
Math. Z., 1960, {\bf 73}, 235--253.

\bibitem{Liebert}
Liebert W.
\emph{Isomorphic automorphism groups of primary abelian groups}. II,
Contemp. Math., 1989, {\bf 87}, 51--59.

\bibitem{Maltsev}   Maltsev A.I. \emph{On elementary properties of linear groups.
Problems of mathematics and mechanics}, Novosibirsk, 1961, 110--132 (in Russian).





\bibitem{My2}
Roizner M.A.
\emph{A criterion of elementary equivalence of automorphism groups of reduced Abelian $p$-groups},
J. Math. Sci., 2013, {\bf 193}, 586--590.

\bibitem{My3} Roizner M.A. \emph{A criterion of elementary equivalence of automorphism groups of unreduced Abelian $p$-groups}, J. Math. Sci., 2014, {\bf 201},  519--526.

\bibitem{Misha-phd} Roizner M.A. \emph{Elementary equivalence of endomorphism rings and automorphism groups of abelian $p$-groups}, PhD thesis, Moscow, Russia, 2013. 

\bibitem{Shelah}
Shelah S.
\emph{Interpreting set theory in the endomorphism semi-group of a free algebra or in the category},
Annales Scientifiques L'universite Clermont, 1976, {\bf 13}, 1--29.

\bibitem{a1} Szmielew W. \emph{Elementary properties of Abelian groups}. 
Fundamenta Mathematica, 1955, {\bf 41}, 203--271.


\bibitem{Tolstyh} Tolstykh~V. \emph{Elementary equivalence of infinite-dimensional
classical groups}, Annals of Pure and Applied Logic, 2000, {\bf 105}, 103--156.

    \end{thebibliography}
\end{document}